\documentclass{amsart}
\usepackage{amssymb,amsthm,amsmath}
\usepackage[all]{xy}

\makeatletter
\renewcommand{\@biblabel}[1]{#1.}
\makeatother

\theoremstyle{definition}
\newtheorem{dfn}{DEFINITION}[section]
\theoremstyle{plain}
\newtheorem{thm}{THEOREM}[section]
\newtheorem{lem}{LEMMA}[section]
\newtheorem{prop}{PROPOSITION}[section]
\newtheorem{cor}{COROLLARY}[section]
\theoremstyle{remark}
\newtheorem{rem}{REMARK}[section]

\newcommand{\ha}{\widehat{\mathfrak{h}}}
\newcommand{\Sy}[1]{{\rm S}( {#1}{}_- )}
\newcommand{\cal}{\mathcal}

\title[The intermediate vertex subalgebras]
{The intermediate vertex subalgebras of the lattice vertex operator algebras}
\author{Kazuya KAWASETSU}
\address{Department of Mathematical Sciences, University of Tokyo, Komaba, Tokyo, 153-8914, Japan.}
\email{kawasetu@ms.u-tokyo.ac.jp}

\begin{document}

\begin{abstract}
A notion of {\it intermediate vertex subalgebras} of lattice vertex operator algebras is introduced, as a generalization of the notion of {\it principal subspaces}.
Bases and the graded dimensions of such subalgebras are given.
As an application, it is shown that the characters of some modules of an intermediate vertex subalgebra between $E_7$ and $E_8$ lattice vertex operator algebras satisfy some modular differential equations.
This result is an analogue of the result concerning the ``hole" of the {\it Deligne dimension formulas} and the {\it intermediate Lie algebra} between the simple Lie algebras $E_7$ and $E_8$.
\end{abstract}


\maketitle

\section{Introduction}

The {\it Deligne exceptional series of simple Lie algebras} is the series
\[
A_1\subset A_2 \subset G_2 \subset D_4 \subset F_4 \subset E_6 \subset E_7 \subset E_8
\]
of simple Lie algebras \cite{D}.
For irreducible components of some tensor products of the adjoint representations of the simple Lie algebras in the above exceptional series,
remarkable dimension formulas, called {\it Deligne dimension formulas}, were established \cite{CdM,D,LaM2}.
They are expressed as rational functions in the dual Coxeter number $h^\vee$.
For example,
\[
\dim \mathfrak{g} = \frac{2(5h^\vee -6)(h^\vee +1)}{h^\vee +6}
\]
and
\[
\dim \mathfrak{g}^{(2)} = \frac{5h^{\vee 2}(2h^\vee+3)(5h^\vee-6)}{(h^\vee+6)(h^\vee+12)}.
\]

When $h^\vee=24$, which intermediates between the dual Coxeter numbers
$18$ of $E_7$ and $30$ of $E_8$, the formulas give integer values $\dim\mathfrak{g}=190$ and 
 $\dim\mathfrak{g}^{(2)}=15504$.
However there is no such a simple Lie algebra.

Later, this ``hole" of the exceptional series was filled in.
In \cite{LaM1}, a Lie algebra $E_{7+1/2}$, which is non-reductive and intermediates between $E_7$ and $E_8$, was constructed,
and the dimension formulas for this algebra were proved.
The Lie algebra $E_{7+1/2}$ is an
{\it intermediate Lie algebra} \cite{LaM1,S1,GZ1}.

The same exceptional series appeared in earlier studies of modular differential equations.
In 1988, Mathur, Mukhi and Sen, in their work of classification of rational conformal field theories ($C_2$-cofinite rational vertex operator algebras (VOAs) of CFT-type) with two characters \cite{MMS},
studied the modular differential equations of the form
\begin{equation}\label{eq:diff1}
\left( q\frac{d}{dq}\right)^2 f(\tau) + 2 E_2(\tau)\left(q\frac{d}{dq}\right)f(\tau) + 180\mu\cdot E_4(\tau)f(\tau)=0.
\end{equation}
Here $\mu$ is a numerical constant, $\tau$ a complex number in the complex upper half-plane $\mathbb{H}$ with $q=e^{2\pi i \tau}$, and $E_k(\tau) (k=2,4,6,\ldots)$ the Eisenstein series
\[
E_k(\tau)=-\frac{B_k}{k!}+\frac{2}{(k-1)!}\sum_{n\geq 1}\frac{n^{k-1}q^n}{1-q^n},
\]
where $B_k$ is the $k$-th Bernoulli number.
(Differential equations equivalent to (\ref{eq:diff1}) were studied by Kaneko and Zagier \cite{Kan3} in number theory.)
By studying (\ref{eq:diff1}), they showed, roughly speaking, that the characters of the rational conformal field theories with two characters are that of the
level one affine VOAs associated to the Deligne exceptional simple Lie algebras.
The list obtained is shown in Table \ref{tb:table1}.
($c$ denotes the central charge, $h$ the non-zero conformal weight, and $\dim V_1$ the dimension of the weight one subspace of such a theory.)
When $\mu=11/900$ and $551/900$, there are solutions of the differential equations (\ref{eq:diff1}) of the form
$
f(\tau)= q^{-c/24} \sum_{n=0}^\infty a_n q^n
$
with $a_0=1$, $a_n\in \mathbb{Z}_{> 0}$  ($n=1,2,3,\ldots$) and $c=2/5$ and $38/5$.
However, according to the Verlinde formula, there are no rational conformal field theories of central charges $c=2/5$ and $38/5$ with two characters.
(In \cite{MMS}, the Virasoro minimal model at $c=-22/5$ was assigned to the case $\mu=11/900$.
The characters agree with the famous {\it Rogers-Ramanujan functions}.)
Note that the value $c=38/5$ intermediates between the central charges $c=7$ of  $L_{1,0}(E_7)$ and $c=8$ of $L_{1,0}(E_8)$.
Here, $L_{1,0}(\mathfrak{g})$ is the level one affine VOA associated to a  simple Lie algebra $\mathfrak{g}$.
Note also that actually $a_1$ is $190$ and agrees with  $\dim E_{7+1/2}=190$.

\begin{table}
\begin{center}
 \label{tb:table1}
\begin{tabular}{|l|r|c|c|l|} \hline
$\mu$ & $\dim V_1$ & $c$ & $h$ & Identification \\ \hline \hline
$11/900$ & $1$ & $2/5$ & $1/5$ &  \\ \hline
$5/144$ & $ 3 $ & $ 1 $ & $1/4$ & $L_{1,0}(A_1)$ \\ \hline
$1/12$ & $ 8 $ & $2$ & $1/3$ & $L_{1,0}(A_2)$\\ \hline
$119/900$ & $14$ & $14/5$ & $2/5$ &$L_{1,0}(G_2)$\\ \hline
$2/9$ & $28$ & $4$ & $1/2$ &$L_{1,0}(D_4)$\\ \hline
$299/900$ & $52$ & $26/5$ & $3/5$ &$L_{1,0}(F_4)$\\ \hline
$5/12$ & $78$ & $6$ & $2/3$ &$L_{1,0}(E_6)$\\ \hline
$77/144$ & $133$ & $7$ & $3/4$ &$L_{1,0}(E_7)$\\ \hline
$551/900$ & $190$ & $38/5$ & $4/5$ &  \\ \hline
$2/3$ & $248$ & $8$ & $(5/6)$ & $L_{1,0}(E_8)$\\ \hline
\end{tabular}
\caption{The characters of the RCFTs with ``two" characters and $h\geq 0$}
\end{center}
\end{table}

Motivated by the above works, we construct an $\mathbb{N}$-graded vertex algebra and a module with the ``characters" satisfying the equation (\ref{eq:diff1}) with $\mu=551/900$, 
by considering the analogy of the intermediate Lie algebras.
We consider the vertex subalgebra
\begin{equation}\label{eq:e712affine}
\langle E_{7+1/2} \rangle_{{\rm v.a.}} \subset L_{1,0}(E_8)
\end{equation}
(and a module).
Here,  $\langle E_{7+1/2}\rangle_{{\rm v.a.}}$ denotes the smallest vertex subalgebra containing $E_{7+1/2}$.

The vertex algebra $\langle E_{7+1/2} \rangle_{{\rm v.a.}}$ is isomorphic to an {\it intermediate vertex subalgebra} of a lattice VOA, which we introduce in this paper.
For the purpose, we establish the formula to describe the graded dimensions of such subalgebras and modules.

The intermediate vertex subalgebra $W(R,S)$ is by definition the vertex subalgebra generated by 
the subset $\{e^\rho, e^{\pm \sigma}, \sigma_{-1}\otimes 1 |  \rho\in R,\sigma\in S\}$ of the lattice VOA $V_L$ associated with an integral lattice $L$.
Here, $R$ and $S$ are disjoint subsets of a $\mathbb{Z}$-basis $B$ of $L$.
(See DEFINITION \ref{sec:def1} and \ref{sec:def2}, for more detail.)

The notion of the intermediate vertex subalgebras is a generalization of the notion of the {\it principal subspaces} introduced by Feigin and Stoyanovsky \cite{FS,SF}.
A principal subspace is the subspace
\[
W(\Lambda)=U(\bar{\mathfrak{n}})\cdot v_\Lambda
\]
of a standard $A_n^{(1)}$-module $L(\Lambda)$, where $\mathfrak{n}$ is the nilradical of a Borel subalgebra of $sl_{n+1}$.
When $n=1$, the graded dimensions of $W(\Lambda_0)$ and $W(\Lambda_1)$ agree with the Rogers-Ramanujan functions.
The notion clearly extends to an arbitrary highest weight module for an affine Lie algebra.
The principal subspaces were studied in \cite{CLM2, CalLM3, G, AKS,CoLM,P1,FFJMM} and others.

Recently, Milas and Penn considered the lattice VOA $V_L$ and the vertex subalgebra
$W_L(B)=\langle e^{\beta_1},\ldots,e^{\beta_n} \rangle _{{\rm v.a.}}$, 
 called the principal subalgebra \cite{MP}.
This is a generalization of the principal subspaces of level one standard modules over simply-laced simple Lie algebras.
Combinatorial bases and the graded dimensions of the subalgebra $W_L(B)$ and some modules were given in \cite{MP}.

Note that the principal subalgebra $W_L(B)$ agrees with the intermediate vertex subalgebra $W(B,\emptyset)$.
By using the bases of the principal subalgebras and modules, we construct combinatorial bases and give the formula to compute the graded dimensions of the intermediate vertex subalgebra $W(R,S)$ and {\it intermediate modules}.

We use the formula to study the character of the vertex subalgebra 
(\ref{eq:e712affine}).
We consider the lattice VOA $V_{E_8}$ associated with the $E_8$ root lattice and the intermediate vertex subalgebra $V_{E_{7+1/2}}=W(\{\alpha_1\},\{\alpha_2,\ldots,\alpha_8\})$.
This vertex algebra is isomorphic to the vertex algebra (\ref{eq:e712affine}) via the isomorphism $V_{E_8}\cong L_{1,0}(E_8)$.
Next, we consider the intermediate module $V_{E_{7+1/2}+\alpha_1}$.
Then we will show that the characters of the subalgebra and module form a basis of the solutions of the modular differential equation (\ref{eq:diff1}) with $\mu=551/900$.

Note that by means of Tuite's result in \cite{T2}, which was motivated by the work of Matsuo \cite{Mat}, our result can be thought of as a conformal field theory version of filling in the ``hole" of the exceptional series.

In section 2, we recall the definition of the lattice VOAs and introduce the notion of the intermediate vertex subalgebras and intermediate modules.
We then describe the structures and give the formula to compute the graded dimensions.
In section 3, we consider the intermediate vertex subalgebra $V_{E_{7+1/2}}$ and the module $V_{E_{7+1/2}+\alpha_1}$.
We compute the characters and show that they form a basis of the solutions of (\ref{eq:diff1}) with $\mu=551/900$.
For the purpose, we decompose the characters into the form of polynomials in the characters of the modules of some VOAs.
In section 4, we prove the structure theorem of the intermediate vertex subalgebras and modules.

\section{The intermediate vertex subalgebras}

\subsection{The setting}
\label{sec:setting}
Let $L$ be a rank $n$ non-degenerate integral lattice with the $\mathbb{Z}$-bilinear form $\langle\cdot , \cdot \rangle: L\times L\rightarrow \mathbb{Z}$.
Let  $B=\{ \beta_1,\ldots,\beta_n\}$ be a $\mathbb{Z}$-basis of $L$.

Put $\mathfrak{h}=\mathbb{C}\otimes_\mathbb{Z} L$. 
Consider the affine central extension $\widehat{\mathfrak{h}}=\mathfrak{h}\otimes \mathbb{C}[t,t^{-1}]\oplus \mathbb{C}{\bf k}$ and its irreducible induced module
\[
M(1)=U(\widehat{\mathfrak{h}}) \otimes_{U(\mathfrak{h} \otimes \mathbb{C}[t]\oplus \mathbb{C}{\bf k})} \mathbb{C},
\]
where $\mathfrak{h}\otimes \mathbb{C}[t] $ acts trivially and ${\bf k}$ acts as $1$ on the one-dimensional module $\mathbb{C}$.
The space $M(1)$ can be identified with the symmetric algebra ${\rm S}(\widehat{\mathfrak{h}}_-)$, where
\[
\widehat{\mathfrak{h}}_-=\mathfrak{h}\otimes t^{-1}\mathbb{C}[t^{-1}].
\]
Consider the corresponding lattice vertex operator (super)algebra (lattice VOA)
\[
V_L\cong M(1) \otimes \mathbb{C}[L]
\]
\cite{B}.
Recall that the vertex operator is given by the following formula:
\[
Y(e^\alpha,x)=\sum_{m\in \mathbb{Z} } (e^\alpha)_m x^{-m-1} = E^- (-\alpha,x) E^+ (-\alpha,x)e_\alpha x^\alpha.
\]
Here
\[
e_\alpha \cdot (h\otimes e^\beta )=\epsilon (\alpha , \beta )h\otimes e^{\alpha+\beta}, \ \ \ e^{\beta} \in \mathbb{C} [L], h\in M(1),
\]
\[
E^-(-\alpha,x)=\exp \left(-\sum_{j<0} \frac{x^{-j} } {j} \alpha_j \right),
\]
and
\[
E^+(-\alpha,x)=\exp \left(-\sum_{j>0} \frac{x^{-j} } {j} \alpha_j \right).
\]

For $B'\subset B$,
we set
\[
L(B')=\bigoplus_{\beta \in B'} \mathbb{Z}\beta
\ \ \ \ \mbox{and} \ \ \ \ L_+(B')=\bigoplus_{\beta \in B'} \mathbb{Z}_{\geq 0} \beta.
\]
Then $L(B')$ is a sublattice, and $L_+(B')$ is a submonoid of $L$.
Furthermore, we set $\mathfrak{h}(B')$, $\widehat{\mathfrak{h}}(B')$ and $\widehat{\mathfrak{h}}(B')_-$
and consider the lattice VOA $V_{L(B')}$ as a vertex operator subalgebra of $V_L$.

Let $R$ and $S$ be disjoint subsets of $B$.
Let $R\sqcup S$ denote the disjoint union.
Put $r=|R|$ and $s=|S|$.
We arrange the indices of the basis so that
$
R=\{\beta_1,\ldots,\beta_r \}$ and $ S=\{ \beta_{r+1},\ldots, \beta_{r+s}\}.
$
We put $\rho_i=\beta_i$ for $1\leq i \leq r$ and $\sigma_j=\beta_{r+j}$ for $1\leq j \leq s$.
That is, 
\[
R=\{ \rho_1,\ldots,\rho_r \} \ \ \ \mbox{and} \ \ \ S= \{ \sigma_1,\ldots,\sigma_s \}.
\]
Then we set
\[
L(R,S)= L_+(R) \oplus L(S).
\]
This is a submonoid of $L$.

Let $\langle A \rangle_{{\rm v.a.}}$ denote the smallest vertex subalgebra containing the subset $A$ of $V_L$.

\begin{dfn}
\begin{em}
\label{sec:def1}
The (weak) {\it intermediate vertex subalgebra} $W(R,S)$ of $V_L$ associated with $(R,S)$ is the vertex subalgebra
\[
W(R,S) = \langle e^{\rho}, e^{\pm \sigma}, \sigma_{-1}\otimes 1 | \rho\in R, \sigma \in S \rangle_{{\rm v.a.}}.
\]
\end{em}
\end{dfn}

We set $W(R)=W(R,\emptyset)$.

Let $L^\circ$ denote the dual lattice of $L$.
Consider a $V_L$-module $V_{L^\circ}=M(1)\otimes \mathbb{C}[L^\circ]$.
Let $\lambda$ be an element of $L^\circ$.

\begin{dfn}
\begin{em}
\label{sec:def2}
The (weak) {\it intermediate module} $W(R,S;\lambda)$ over $W(R,S)$ is the cyclic $W(R,S)$-module
\[
W(R,S;\lambda) = W(R,S) \cdot e^\lambda \subset V_{L^\circ}.
\]
\end{em}
\end{dfn}

We set $W(R;\lambda)=W(R,\emptyset;\lambda)$.

\begin{rem}
\begin{em}
Put $L'=L(R)$ and $L''=L(S)$.
The intermediate vertex subalgebra $W(\emptyset,S)$ agrees with the lattice VOA $V_{L''}$ associated with the lattice $L''$.
On the other hand, the intermediate vertex subalgebra $W(R)=W(R,\emptyset)$ is the principal subalgebra $W_{L'}(R)$,
 and the intermediate modules $W(R;\lambda)=W(R,\emptyset;\lambda)$ agree with the principal subspaces $W_{L'+\lambda}(R)$ \cite{MP}.
Moreover, if $L'$ is an ADE root lattice, $W(R)$ and $W(R;\lambda_i)$ correspond to the level one principal subspaces 
$W (\Lambda_0)$ and $W(\lambda_i)$, studied in \cite{CalLM3}.
Here, $\lambda_i$ $(i=1,\ldots,n)$ are the fundamental weights of the root system of $L'$.
\end{em}
\end{rem}

Now let us define some (bi-)gradings on the above vertex algebras and modules.
Put $V=V_{L^\circ}$.
First, take the conformal vector $\omega$ in $V_L$ and the stress-energy tensor
\[
Y(\omega,z)=T(z)=\sum_{i\in \mathbb{Z}}L(i) z^{-i-2}.
\]
Then, the operator $L(0)$ on $V$ is diagonalizable, and the eigenvalues are rational numbers.
 Let $V_r$ denote the $r$-eigenspace.
The grading $V=\bigoplus_r V_r$ is called the {\it weight grading}.
For $v\in V_r$, we call $r$ the {\it weight} of $v$ and write $r={\rm wt}(v)$.
For any vector subspace $X$ of $V$, we set $X_r=X\cap V_r$.
Then the vector spaces $V_L$, $W(R,S)$ and $W(R,S;\lambda)$ are graded vector subspaces of $V$.
We call the restricted grading the weight grading.

Next,
take an element $\tau$ of $L^\circ$.
Consider the subspace
\[
V^\tau=M(1)\otimes e^\tau.
\]
Then consider the subspaces
\begin{eqnarray*}
X^\tau = X\cap V^\tau, \ \ \ \ (X=V_L, W(R,S)),
\end{eqnarray*}
and
\begin{equation*}
(W(R,S;\lambda))^\tau=W(R,S;\lambda)\cap V^{\tau+\lambda}.
\end{equation*}
Let $X$ be one of the vector spaces $V, V_L, W(R,S)$ and $W(R,S;\lambda)$.
The grading $X=\bigoplus_\tau X^\tau$ is called the {\it charge-grading}.
For any $v\in X^\tau$, we call $\tau$ the {\it charge} of $v$.
If $X^\tau\neq 0$, we call $\tau$ a {\it charge} of $X$.
Note that the set of the charges of $V_L$ agrees with $L$, and that of $W(R,S)$ and $W(R,S;\lambda)$ agree with $L(R,S)$.

Note that for any charge $\tau$, the subspace $X^\tau$ is $L(0)$-invariant, and the weight grading and charge grading are compatible. 
Consider the bi-grading $X=\bigoplus_{r,\tau} X^\tau_r$, where $X^\tau_r$ denote the subspace $X^\tau \cap X_r$.
We call this the {\it charge and weight grading}.

Note that as the bi-graded vector spaces, $W(R,S;0)=W(R,S)$.

\subsection{Structure of the intermediate vertex subalgebras and modules}

Recall that $W(R)=W(R,\emptyset)$ is a principal subalgebra \cite{MP}.
The following theorem is our main result.

\begin{thm}\label{sec:str}
The module $W(R,S;\lambda)$ decomposes as an $\Sy{\ha(S)}\otimes W(R)$-module into the form
\begin{equation}\label{eqn:streqn1}
W(R,S;\lambda) \cong \Sy{\ha(S)}\otimes \bigoplus_{\delta \in L(S)} W(R;\delta+\lambda).
\end{equation}
\end{thm}

We prove the theorem in \S 4.

Now, we construct a basis of $W(R,S;\lambda)$ using the theorem and the result of \cite{MP}.
Write $r=|R|$ and $s=|S|$.
Put $\rho_i=\beta_i$ for $1\leq i \leq r$,  and  $\sigma_j=\beta_{r+j}$ for $1\leq j \leq s$.
Let $k_1,\ldots,k_r$ be non-negative integers.
For $1\leq i \leq r$, consider the set
\begin{eqnarray*}
 M_i(R;\lambda;k_1,\ldots,k_r) &=& 
\{
\left(m_{k_i},\ldots,m_1\right) \in \mathbb{Z}^{k_i} | \\
&& \ \ \ \ \ \ m_1 \leq -1 - \sum_{l=1}^{i-1} k_l \langle \rho_i, \rho_l \rangle - \langle \rho_i, \lambda \rangle, \\
&& \ \ \ \ \ \ m_{j+1}\leq m_j - \langle \rho_i, \rho_i \rangle \ 
(1\leq j \leq k_i-1)
\}.
\end{eqnarray*}
For each sequence $\mu=(m_k,\ldots,m_1)$ of integers and $\beta\in B$, set $\varepsilon^\beta_\mu=(e^\beta)_{m_k}\ldots(e^\beta)_{m_1}$.
Consider the set
\begin{eqnarray*}
{\cal B}(R;\lambda;k_1,\ldots,k_r) &=& \{
\varepsilon^{\rho_r}_{\mu_r}\ldots \varepsilon^{\rho_1}_{\mu_1}.  e^{\lambda} | 
 \mu_i\in M_i(R;\lambda;k_1,\ldots,k_r) \ (1\leq i \leq r)
\}.
\end{eqnarray*}
Note that the elements of the set ${\cal B}(R;\lambda;k_1,\ldots,k_r)$ have the charge $k_1\rho_1+\cdots+k_r\rho_r$.

The following lemma is the result of \cite{MP}.

\begin{lem}(\cite{MP}, Corollary 4.8.) \label{sec:lemmp}
If $k_1,\ldots,k_r$ are non-negative integers,
then the set ${\cal B}(R;\lambda;k_1,\ldots,k_r)$ is a $\mathbb{C}$-basis of the vector space 
$(W(R;\lambda))^\rho$, where $\rho={k_1\beta_1+\cdots+k_r\beta_r}$.
\end{lem}

Note that for each $\tau\in L(R,S)$, the subspace $(W(R,S;\lambda))^\tau$ is a free $\Sy{\ha(S)}$-module.
By THEOREM \ref{sec:str} and LEMMA \ref{sec:lemmp}, we obtain the following corollary.

\begin{cor}\label{sec:cor1}
If $\delta$ is an element of $L(S)$, and $k_1,\ldots,k_r$ are non-negative integers, 
then the set ${\cal B}(R;\lambda+\delta;k_1,\ldots,k_r)$ is a basis of the free $\Sy{\ha(S)}$-module 
$(W(R,S;\lambda))^\tau$, where $\tau={k_1\rho_1+\cdots+k_r\rho_r+\delta}$.
\end{cor}

Consider the set
\begin{eqnarray*}
{\cal B}(\Sy{\ha(S)}) &=&
\{ (\sigma_1)_{i_1^1}\ldots(\sigma_1)_{i_1^{l_1}}\ldots(\sigma_s)_{i_s^1}\ldots(\sigma_s)_{i_s^{l_s}} | \\
&& \ \ \ \ \ l_1,\ldots,l_s \geq 0, i_j^1 \leq \cdots \leq i_j^{l_j} \leq -1 \ (1\leq j \leq s)\}.
\end{eqnarray*}
Then ${\cal B}(\Sy{\ha(S)})$ is a basis of $\Sy{\ha(S)}$.
Consider the set
\[
{\cal B}(R,S;\lambda)={\cal B}(\Sy{\ha(S)})\otimes \coprod_{\delta\in L(S), k_1, \ldots, k_r\geq 0} {\cal B}(R;\lambda+\delta;k_1,\ldots,k_r).
\]
Here, for vector spaces $P$ and $Q$ and subsets $X\subset P$ and $Y\subset Q$, we  denote by $X\otimes Y$ the set
$
\{ x\otimes y \in P \otimes Q | x\in X, y\in Y\}.
$

Then we obtain the following corollary.

\begin{cor}\label{sec:cor11}
The following hold.
\begin{description}
\item[(i)]
If $\delta$ is an element of $L(S)$, and $k_1,\ldots,k_r$ are non-negative integers,
then the set ${\cal B}(\Sy{\ha(S)})\otimes {\cal B}(R;\lambda+\delta;k_1,\ldots,k_r)$
is a $\mathbb{C}$-basis of the vector space $W(R,S;\lambda)^\tau$, where $\tau={k_1\rho_1+\cdots+k_r\rho_r+\delta}$.
\item[(ii)]
The set ${\cal B}(R,S;\lambda)$ is a $\mathbb{C}$-basis of the vector space  $W(R,S;\lambda)$.
\end{description}
\end{cor}

This is a generalization of the result of \cite{MP}.

\subsection{Graded dimensions of the intermediate vertex subalgebras and modules}

Recall that $V_{L^\circ}$ and our subspaces carry the bi-gradings: the charge and weight gradings.

\begin{dfn}
\begin{em}
The {\it graded dimension} $\chi _{W(R,S;\lambda) }$ of $W(R,S;\lambda)$ is
\[
\chi _{W(R,S;\lambda) } ({\bf x};q) = \sum_{\tau \in L^\circ, r\in \mathbb{Q}} \dim_{\mathbb{C}} \left( \left(W(R,S;\lambda) \right)^\tau_r \right) q^r x_1^{k_1} \cdots x_n^{k_n},
\]
where $k_1\beta_1 +\cdots +k_n \beta_n=\tau$.
\end{em}
\end{dfn}

To compute the graded dimension, consider the symbols
\[
(q)_k=(q; q)_k=(1-q)\cdots(1-q^k), \ \ \  (k\geq 1),
\]
$(q)_0=1$,
and
$(q)_{\infty}=\prod_{i=1}^\infty (1-q^i)$.
Here $
(a;q)_k=(1-a)(1-aq)\cdots (1-aq^{k-1})
$
is the $q$-Pochhammer symbol.
Recall that $1/(q)_k$ agrees with the generating function of the partitions into parts not greater than $k$,
 therefore agrees with the generating function of the partitions into at most $k$ parts.
Recall further that
\begin{equation*}
\frac{1}{(q)_{\infty}} = \frac{1}{\varphi(q)} = \sum_{k=0} ^{\infty} p(k) q^k.
\end{equation*}
Here, $\varphi(q)$ is the Euler function and $p(k)$ is the number of the un-restricted partitions of an integer $k$.
Consider the Gram matrix $A=(\langle \beta_i,\beta_j \rangle)_{i,j}$.
Put $r=|R|$ and $s=|S|$.
Since $\lambda\in L^\circ$, $\lambda$ has the form $l_1\beta_1+\cdots +l_n\beta_n$ with $l_1,\ldots, l_n \in \mathbb{Q}$.
Put ${\bf l}=(l_1,\ldots,l_n)$.

\begin{thm} \label{sec:char}
The graded dimension of $W(R,S;\lambda)$ is given by
\[
\chi _{W(R,S;\lambda) } ({\bf x};q) 
= \sum_{k_1,\ldots,k_r \geq 0, k_{r+1},\ldots,k_{r+s} \in \mathbb{Z}}
\frac{q^{ \frac{ ({\bf k}+{\bf l}) \cdot A \cdot ({\bf k}+{\bf l}) } {2} } } {  (q)_{k_1} \cdots (q)_{k_{r}}\cdot (q)_{\infty}^s } 
x_1^{k_1} \cdots x_{r+s}^{k_{r+s}},
\]
where ${\bf k}=(k_1,\ldots,k_{r+s},0,\ldots,0)$.
\end{thm}

\begin{proof}
The assertion follows from  COROLLARY \ref{sec:cor11}
and the relation
\[
{\rm wt}(e^{\tau+\lambda})=\frac{ \langle \tau+\lambda, \tau+\lambda \rangle} {2} = \frac{ ({\bf k}+{\bf l}) \cdot A \cdot ({\bf k}+{\bf l}) } {2},
\]
where $\tau=k_1\beta_1+\cdots +k_n\beta_n$.
\end{proof}

For later use, we set
$
\chi'_{W(R,S;\lambda)}({\bf x};q)=q^{-\langle \lambda,\lambda \rangle /2}\chi_{W(R,S;\lambda)}({\bf x};q).
$

\section{Applications (Filling in the ``hole" of the characters of the ``exceptional" series) }

As an application of Theorem \ref{sec:char}, we will study the intermediate vertex subalgebra $V_{E_{7+1/2}}$ between the lattice VOAs $V_{E_7}$ and $V_{E_8}$.

\subsection{A well-known example: the intermediate vertex subalgebra $V_{A_{1/2}}$}

First we consider the well-known result as an example of the intermediate vertex subalgebras.

Let $A_1$ be a root lattice of type $A_1$ with the $\mathbb{Z}$-bilinear form $\langle, \rangle$,
and let $A_1^\circ$ be the dual lattice.
Consider the lattice VOA  $V_{A_1}$ associated with $A_1$.
Let $\alpha$ be a simple root and $\omega$ the fundamental weight.
Set $\Delta_1=\{\alpha\}$.
 
Consider the intermediate vertex subalgebra $V_{A_{1/2}}=W(\Delta_1,\emptyset)$  and the intermediate module $V_{A_{1/2}+\omega}=W(\Delta_1,\emptyset;\omega)$.

Note that they agree with the principal subspaces of the basic representations of the affine Lie algebra $A_1^{(1)}$.
 The graded dimensions were described in \cite{FS,SF}:
\[
\chi_{V_{A_{1/2}}}(x;q)=\sum_{k=0}^\infty \frac{q^{k^2}} {(q)_k}x^k
\]
and
\[
\chi'_{V_{A_{1/2}+\omega} } (x;q) = \sum_{k=0}^\infty \frac{q^{k^2+k} } {(q)_k}x^k.
\]
Set
\[
c=\frac{2}{5} \ \ \ \ \mbox{and} \ \ \ \ 
h=\frac{1}{5},
\]
as in the row of $\mu=11/900$ in Table \ref{tb:table1}.
We define the characters of $V_{A_{1/2}}$ and $V_{A_{1/2}+\omega}$ to be
\[
Z(V_{A_{1/2}};\tau)= q^{-c/24} \chi_{V_{A_{1/2}}}(1;q)
\]
and
\[
Z(V_{A_{1/2}+\omega};\tau)=q^{h-c/24} \chi'_{V_{A_{1/2}+\omega}}(1;q),
\]
where $\tau \in \mathbb{H}$ and $q=e^{2\pi i \tau}$.

Then the vector space spanned by the characters $Z(V_{A_{1/2}};\tau)$ and 
$Z(V_{A_{1/2}+\omega};\tau)$ is invariant under modular transformations and is the space of the solutions of the modular differential equation (\ref{eq:diff1}) with $\mu=11/900$:
\[
\left( q\frac{d}{dq}\right)^2 f(\tau) + 2 E_2(\tau)\left(q\frac{d}{dq}\right)f(\tau)-\frac{11}{5}E_4(\tau)f(\tau)=0.
\]
In fact, we have
\begin{equation} \label{eq:rr1}
Z\left(V_{A_{1/2}};\tau\right)=q^{-1/60} \sum_{k=0}^\infty \frac{q^{k^2}} {(q)_k}
\end{equation}
and
\begin{equation}\label{eq:rr2}
Z\left(V_{A_{1/2}+\omega} ;\tau\right) = q^{11/60} \sum_{k=0}^\infty \frac{q^{k^2+k} } {(q)_k}.
\end{equation}
The RHS of (\ref{eq:rr2}) agrees with the character of the Virasoro minimal model $L(-22/5,0)$ at $c=-22/5$,
and the RHS of (\ref{eq:rr1}) agrees with that of the unique non-identity irreducible module $L(-22/5,-1/5)$.

Hence,
\[
Z\left(V_{A_{1/2}};\tau\right)=Z\left(L\left(-22/5,-1/5\right);\tau\right),
\]
and
\[
Z\left(V_{A_{1/2}+\omega} ;\tau\right)=Z\left(L\left(-22/5,0\right);\tau\right).
\]
Therefore the characters of the intermediate subalgebra $V_{A_{1/2}}$ and the module $V_{A_{1/2}+\omega}$ form a basis of the space of the solutions of a modular differential equation.
The differential equation can be standardly deduced using a singular vector of the Virasoro Verma module $V(-22/5,0)$.
It coincides with the above differential equation. (cf: \cite{M} Theorem 6.1.) Thus the assertion holds.

\begin{rem}
\begin{em}
 Functions (\ref{eq:rr1}) and (\ref{eq:rr2}) are the Rogers-Ramanujan functions.
 These functions arise naturally  in solutions of the Regime I of Baxter's Hard Hexagon model \cite{Bax, A2}.
\end{em}
\end{rem}

For later use, put $\phi_1(\tau)=\eta(\tau)^{2/5} Z(V_{A_{1/2}})$ and $\phi_2(\tau)=\eta(\tau)^{2/5} Z(V_{A_{1/2}+\omega})$.
Here, $\eta(\tau)$ is the Dedekind eta function.
Note that the functions $\phi_1$ and $\phi_2$ are holomorphic modular forms of weight $1/5$ (with a suitable multiplier system) on the congruence subgroup $\Gamma(5)$ \cite{BKMS,I,Kan}.

\subsection{The intermediate vertex subalgebra $V_{E_{7+1/2}}$}

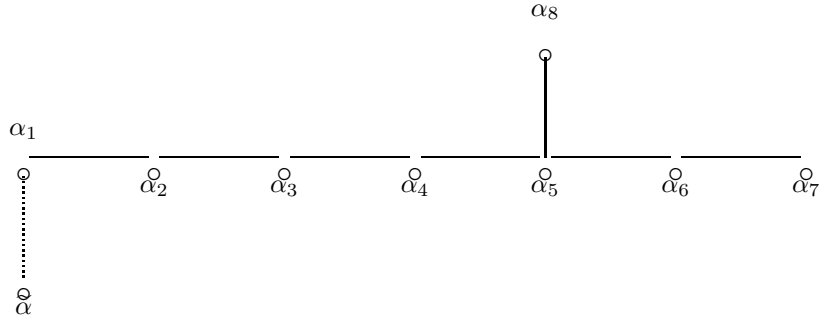
\begin{figure}[htbp]
\begin{center}
\def\maru{\lower7pt\hbox{\Large\hskip-0.0pt$\circ\hskip-0.0pt$}}
$$
\xymatrix@M-3.77pt@C40pt@R5pt{
&&&&\alpha_8\cr
&&&&\maru\ar@{-}[ddd]\cr
\vbox to 20pt{}\cr
\alpha_1\cr
\maru\ar@{-}[r]&\maru\ar@{-}[r]&\maru\ar@{-}[r]&\maru\ar@{-}[r]&\maru\ar@{-}[r]&\maru\ar@{-}[r]&\maru\cr
&\alpha_2&\alpha_3&\alpha_4&\alpha_5&\alpha_6&\alpha_7\cr
\vbox to 20pt{}\cr
\maru\ar@{.}[uuu]\cr
\widetilde{\alpha}\cr
\vbox to 20pt{}\cr
}
$$
\caption{Dynkin diagram of $E_8$}
\label{fig:fig1}
\end{center}
\end{figure}

Let $E_8$ be a root lattice of type $E_8$ with the $\mathbb{Z}$-bilinear form $\langle , \rangle$.
Let $\alpha_1,\ldots,\alpha_8$ be simple roots of $E_8$ and denote the highest root by $\tilde{\alpha}$, as illustrated in Figure \ref{fig:fig1}.
Note that a sublattice $E_7=\langle \alpha_2,\ldots,\alpha_8 \rangle \subset E_8$ is a root lattice of type $E_7$.
Set $\Delta_8=\{\alpha_1,\ldots,\alpha_8\}$.
Consider the lattice VOA $V_{E_8}$  associated with $E_8$ and the lattice subVOA $V_{E_7}\subset V_{E_8}$ associated with $E_7$.

Set $R=\{\alpha_1\}$ and $S=\{\alpha_2,\ldots,\alpha_8\}$.
Consider the intermediate vertex subalgebra $V_{E_{7+1/2}}=W(R,S)$ of $V_{E_8}$ associated with $(R,S)$.
That is,
\[
V_{E_{7+1/2}}= \langle e^{\alpha_1}, e^{\pm \alpha_2},\ldots,e^{\pm \alpha_8}, \alpha_2,\ldots,\alpha_8 \rangle_{{\rm v.a.}} \subset V_{E_8}.
\]

As in the previous subsection, we will fill in the ``hole" of the ``exceptional series" using the intermediate vertex subalgebra $V_{E_{7+1/2}}$.
For the purpose, we need to take a ``non-identity" intermediate module with the character satisfying the modular differential equation.
If we consider the ways of the theories of the intermediate Lie algebras and modules, we should take it in the ``non-identity irreducible module" of $V_{E_8}$.
But such a module does not exist since $E_8$ is a unimodular lattice, that is, $E_8=E_8^\circ$.
Therefore we try to take it in the identity module $V_{E_8}$. It will go well.

Consider the element $\alpha_1\in E_8$ and the intermediate module $W(R,S;\alpha_1)$.
We set $V_{E_{7+1/2}+\alpha_1}=W(R,S;\alpha_1)$.

Set
\[
c=\frac{38}{5}
\ \ \ \ \mbox{and} \ \ \ \ h=\frac{4}{5},
\]
as in the row of $\mu=551/900$ in Table \ref{tb:table1}.

We define the characters of $V_{E_{7+1/2}}$ and $V_{E_{7+1/2}+\alpha_1}$ to be
\[
Z(V_{E_{7+1/2}};\tau)= q^{-c/24} \chi_{V_{E_{7+1/2}}}(1,\ldots,1;q)
\]
and
\[
Z(V_{E_{7+1/2}+\alpha_1};\tau)=q^{h-c/24} \chi'_{V_{E_{7+1/2}+\alpha_1}}(1,\ldots,1;q),
\]
where $\tau \in \mathbb{H}$ and $q=e^{2\pi i \tau}$.

Then the vector space spanned by the characters $Z(V_{E_{7+1/2}};\tau)$ and \\ $Z(V_{E_{7+1/2}+\alpha_1};\tau)$ is invariant under modular transformations and is the space of the solutions of the modular differential equation (\ref{eq:diff1}) with $\mu=551/900$:
\begin{equation} \label{eq:diff712}
\left( q\frac{d}{dq}\right)^2 f(\tau) + 2 E_2(\tau)\left(q\frac{d}{dq}\right)f(\tau)-\frac{551}{5}E_4(\tau)f(\tau)=0.
\end{equation}

Let us prove the assertion.
Denote by $(\omega_2,\ldots,\omega_8)$  ($\omega_i \in E_7^\circ$ ($i=2,\ldots,8$))  the dual basis of the basis $(\alpha_2,\ldots,\alpha_8)$ of the sublattice $E_7$.
Consider the Virasoro minimal model $L(-3/5,0)$ at $c=-3/5$ and the modules $L(-3/5,h)$ with the conformal weights $h=0,3/4, 1/5$ and $-1/20$.

First, we show that the characters satisfy the following equalities:
\begin{eqnarray} \label{eq:zv712}
Z\left(V_{E_{7+1/2}};\tau\right)&=&Z\left(V_{E_7};\tau\right)\cdot Z\left(L\left(-3/5,-1/20\right);\tau\right) + \nonumber \\
 && Z\left(V_{E_7+\omega_2};\tau\right) \cdot Z\left(L\left(-3/5,1/5\right);\tau\right)
\end{eqnarray}
and
\begin{eqnarray} \label{eq:zl712}
Z\left(V_{E_{7+1/2}+\alpha_1};\tau\right)&=&Z\left(V_{E_7};\tau\right)\cdot Z\left(L\left(-3/5,3/4\right);\tau\right) + \nonumber \\
&& Z\left(V_{E_7+\omega_2};\tau\right) \cdot Z\left(L\left(-3/5,0\right);\tau\right).
\end{eqnarray}
By Theorem \ref{sec:char}, the characters are
\begin{equation} \label{eq:zv712pre}
Z\left(V_{E_{7+1/2}};\tau\right)=q^{-19/60} \cdot \sum_{ k_1\geq 0, k_2,\ldots,k_8 \in \mathbb{Z}} \frac{q^{ \frac{ {\bf k}\cdot M_8 \cdot {\bf k}^{\rm T} }{2} } }{(q)_{k_1} (q)_{\infty}^7 }
\end{equation}
and
\begin{equation} \label{eq:zl712pre}
Z\left(V_{E_{7+1/2}+\alpha_1};\tau\right)=q^{29/60}\cdot \sum_{ k_1\geq 1, k_2,\ldots,k_8 \in \mathbb{Z}} \frac{q^{ \frac{ {\bf k}\cdot M_8 \cdot {\bf k}^{\rm T} }{2} } }{(q)_{k_1-1} (q)_{\infty}^7},
\end{equation}
where ${\bf k}=(k_1,...,k_8)$.
Here $M_8$ is the Cartan matrix $M_8=(\langle \alpha_i, \alpha_j \rangle)_{(i,j=1,\ldots,8)}$ of $E_8$.
In (\ref{eq:zv712pre}), ${\bf k}=(k_1,\ldots,k_8)$ implies the charge $k_1\alpha_1+\cdots +k_8\alpha_8$ of $V$.
Note that the set of the charges of $V_{E_{7+1/2}}$ agrees with the set $L(R,S)=\{k_1\alpha_1+\cdots +k_8\alpha_8 | k_1\geq 0, k_2,\ldots,k_8 \in \mathbb{Z}\}$.

Recall (cf. \cite{K2}) that the highest root $\tilde{\alpha}$ of $E_8$ is
\[
\tilde{\alpha}= 2\alpha_1+3\alpha_2+4\alpha_3+5\alpha_4+6\alpha_5+4\alpha_6+2\alpha_7+3\alpha_8.
\]
Set
\[
L(R,S)^{\rm even}=\{k_2\alpha_2+\cdots +k_8\alpha_8+ k\tilde{\alpha} | k_2,\ldots,k_8 \in \mathbb{Z}, k\geq 0\}
\]
and
\[
L(R,S)^{\rm odd}=\{ k_2\alpha_2+\cdots+k_8\alpha_8+k\tilde{\alpha} +\alpha_1 | k_2,\ldots,k_8 \in \mathbb{Z}, k\geq 0\}.
\]
Then $L(R,S)=L(R,S)^{\rm even} \sqcup L(R,S)^{\rm odd}$.

Let $k_2,\ldots,k_8$ be integers and $k$ a non-negative integer.
Set $\mu=k_2\alpha_2+\cdots+k_8\alpha_8+k\tilde{\alpha}$ and $\nu=k_2\alpha_2+\cdots+k_8\alpha_8+k\tilde{\alpha}+\alpha_1$.
Then $\mu\in L(R,S)^{\rm even}$, and $\nu\in L(R,S)^{\rm odd}$.
Set $\beta=k_2\alpha_2+\cdots+k_8\alpha_8$ and ${\bf k}'=(k_2,\ldots,k_8)$.
Then $\beta$ belongs to $L(S)=E_7$.
Considering the extended Dynkin diagram of $E_8$ (Figure \ref{fig:fig1}), we obtain
\begin{eqnarray*}
\langle \beta,\beta\rangle &=& {\bf k}' \cdot M_7 \cdot {\bf k}'^{{\rm T}}, \\
\langle \beta, \tilde{\alpha} \rangle &=&0, \\
\langle \beta, \alpha_1 \rangle &=& \langle k_2\alpha_2, \alpha_1 \rangle =-k_2
\end{eqnarray*}
and
\begin{equation*}
\langle \tilde{\alpha}, \alpha_1 \rangle = \langle 3\alpha_2+2\alpha_1, \alpha_1 \rangle = 1.
\end{equation*}
Here $M_7$ is the Cartan matrix $(\langle \alpha_i, \alpha_j \rangle)_{(i,j=2,\ldots,8)}$ of $E_7$.
Then we have
\begin{eqnarray*}
\frac{\langle \mu, \mu \rangle}{2} 
= \frac{ {\bf k}' \cdot M_7 \cdot {\bf k}'^{{\rm T}} }{2} + k^2
\end{eqnarray*}
and
\begin{eqnarray*}
\frac{\langle \nu, \nu \rangle}{2} 
= \frac{ {\bf k}' \cdot M_7 \cdot {\bf k}'^{{\rm T}} }{2} - k_2+ k^2+ k + 1.
\end{eqnarray*}
Therefore we have
\begin{eqnarray}
&& Z\left(V_{E_{7+1/2}};\tau\right) \nonumber \\ &=&
 q^{-19/60} \sum_{k\geq 0, k_2,\ldots, k_8 \in \mathbb{Z}} \frac {q^{\frac{ {\bf k}' \cdot M_7 \cdot {\bf k}'^{{\rm T}}}{2} + k^2}} {(q)^7_\infty (q)_{2k}} + \nonumber \\
&&q^{-19/60} \sum_{k\geq 0, k_2,\ldots, k_8 \in \mathbb{Z}} \frac {q^{\frac{ {\bf k}' \cdot M_7 \cdot {\bf k}'^{{\rm T}}} {2} - k_2+k^2+k+1}} {(q)^7_\infty (q)_{2k+1}} \nonumber \\
&=& \left( q^{-7/24}\sum_{k_2,\ldots, k_8\in \mathbb{Z}} \frac {q^{\frac{ {\bf k}' \cdot M_7 \cdot {\bf k}'^{{\rm T}}}{2}}} {(q)_\infty^7} \right)
 \cdot \left( q^{-1/40} \sum_{k\geq 0} \frac {q^{k^2}} {(q)_{2k}} \right) + \nonumber \\
 && \left( q^{11/24} \sum_{k_2,\ldots, k_8 \in \mathbb{Z}} \frac {q^{\frac{ {\bf k}' \cdot M_7 \cdot {\bf k}'^{{\rm T}}}{2} -\langle \beta,\omega_2\rangle }} {(q)_\infty^7} \right)
 \cdot \left( q^{-9/40} \sum_{k\geq 0} \frac {q^{k^2+k}} {(q)_{2k+1}} \right), \nonumber \\
\label{eq:bhh1}
\end{eqnarray}
where ${\bf k}'= (k_2,\ldots, k_8)$, and $\beta=k_2\alpha_2+\cdots +k_8\alpha_8$.
Similarly, we have
\begin{eqnarray}
&& Z\left(V_{E_{7+1/2}+\alpha_1};\tau\right) \nonumber \\ 
&=& \left( q^{-7/24}\sum_{k_2,\ldots,k_8 \in \mathbb{Z}} \frac {q^{\frac{ {\bf k}' \cdot M_7 \cdot {\bf k}'^{{\rm T}}}{2}}} {(q)_\infty^7} \right)
 \cdot \left( q^{-9/40} \sum_{k\geq 1} \frac {q^{k^2}} {(q)_{2k-1}} \right) + \nonumber \\
 && \left( q^{11/24} \sum_{k_2,\ldots, k_8 \in \mathbb{Z}} \frac {q^{\frac{ {\bf k}' \cdot M_7 \cdot {\bf k}'^{{\rm T}}}{2} -\langle \beta,\omega_2\rangle }} {(q)_\infty^7} \right)
 \cdot \left( q^{1/40} \sum_{k\geq 0} \frac {q^{k^2+k}} {(q)_{2k}} \right), \nonumber \\
\label{eq:bhh2}
\end{eqnarray}
where ${\bf k}'=(k_2,\ldots, k_8)$, and $\beta=k_2\alpha_2+\cdots +k_8\alpha_8$.

It is known \cite{KKMM} that the four functions appearing in the above equalities agree with the characters of the Virasoro minimal model at $c=-3/5$:
\begin{eqnarray*}
q^{-1/40}\sum_{k\geq 0} \frac {q^{k^2}} {(q)_{2k}} &=& Z\left(L\left(-3/5,-1/20\right);\tau\right), \\
q^{9/40}\sum_{k\geq 0} \frac {q^{k^2+k}} {(q)_{2k+1}} &=& Z\left(L\left(-3/5,1/5\right);\tau\right), \\
q^{-9/40}\sum_{k\geq 1} \frac {q^{k^2}} {(q)_{2k-1}} &=& Z\left(L\left(-3/5,3/4\right);\tau\right)
\end{eqnarray*}
and
\begin{eqnarray*}
q^{1/40}\sum_{k\geq 0} \frac {q^{k^2+k}} {(q)_{2k}} &=& Z\left(L\left(-3/5,0\right);\tau\right).
\end{eqnarray*}
Furthermore, by the definition of the characters of the lattice VOAs, we have
\begin{eqnarray*}
q^{-7/24}\sum_{k_2,\ldots,k_8 \in \mathbb{Z}} \frac {q^{\frac{ {\bf k}' \cdot M_7 \cdot {\bf k}'^{{\rm T}}}{2}}} {(q)_\infty^7} = Z\left(V_{E_7};\tau\right)
\end{eqnarray*}
and
\begin{eqnarray*}
q^{11/24} \sum_{k_2,\ldots, k_8 \in \mathbb{Z}} \frac {q^{\frac{ {\bf k}' \cdot M_7 \cdot {\bf k}'^{{\rm T}}}{2} -\langle \beta,\omega_2\rangle }} {(q)_\infty^7} = Z\left(V_{E_7-\omega_2};\tau\right),
\end{eqnarray*}
where ${\bf k}'=(k_2,\ldots, k_8)$ and $\beta=k_2\alpha_2+\cdots+ k_8\alpha_8$.
Since
\[
-\omega_2=-\frac{1}{2}(3\alpha_2+4\alpha_3+5\alpha_4+6\alpha_5+4\alpha_6+2\alpha_7+3\alpha_8)
\]
 is a representative of the unique non-zero element of $E_7^\circ/E_7\cong \mathbb{Z}/2\mathbb{Z}$,
the module $V_{E_7-\omega_2}$ (also, $V_{E_7+\omega_2}$) is the unique non-identity irreducible module of the lattice VOA $V_{E_7}$ up to equivalence.
Thus we have (\ref{eq:zv712}) and (\ref{eq:zl712}).

\begin{rem}
\begin{em}
The four functions
\[
\sum_{k\geq 0} \frac {q^{k^2}} {(q)_{2k}}, \ \ \ \sum_{k\geq 0} \frac {q^{k^2+k}} {(q)_{2k+1}}, \ \ \ \sum_{k\geq 1} \frac {q^{k^2}} {(q)_{2k-1}} \ \ \mbox{and} \ \
\sum_{k\geq 0} \frac {q^{k^2+k}} {(q)_{2k}}
\]
appearing in equalities (\ref{eq:bhh1}) and (\ref{eq:bhh2}) arise naturally in solutions of the Regime IV of Baxter's Hard Hexagon model \cite{Bax, A2}.
\end{em}
\end{rem}

By the theories of VOAs (or theories of infinite dimensional Lie algebras \cite{K2,W},) we obtain the rules of modular transformations for the characters of the VOAs $V_{E_7}$ and $L(-3/5,0)$:

\begin{eqnarray}
\begin{pmatrix}
Z\left(V_{E_7}; -1/\tau\right) \\
Z\left(V_{E_7+\omega_2};-1/\tau\right)
\end{pmatrix}
=\frac{1}{\sqrt{2}}
\begin{bmatrix}
1 & 1 \\
1 & -1
\end{bmatrix}
\begin{pmatrix}
Z\left(V_{E_7};\tau\right) \\
Z\left(V_{E_7+\omega_2};\tau\right)
\end{pmatrix},\label{eqn:mod1}
\end{eqnarray}
\begin{eqnarray}
\begin{pmatrix}
Z\left(V_{E_7};\tau+1\right) \\
Z\left(V_{E_7+\omega_2}; \tau + 1 \right)
\end{pmatrix}
=
\begin{bmatrix}
e^{2\pi i(-7/24)} & 0 \\
0& e^{2\pi i(11/24)}
\end{bmatrix}
\begin{pmatrix}
Z\left(V_{E_7}; \tau\right) \\
Z\left(V_{E_7+\omega_2}; \tau\right)
\end{pmatrix},\label{eqn:mod2}
\end{eqnarray}

\begin{eqnarray}
&&
\begin{pmatrix}
Z(L(-3/5,0);-1/\tau) \\
Z(L(-3/5,3/4);-1/\tau) \\
Z(L(-3/5,1/5);-1/\tau) \\
Z(L(-3/5,-1/20);-1/\tau)
\end{pmatrix}
\nonumber \\
&=&
\sqrt{\frac{2}{5}}
\begin{bmatrix}
\sin(2\pi/5) & -\sin(2\pi/5) & -\sin(\pi/5)& \sin(\pi/5) \\
-\sin(2\pi/5) & -\sin(2\pi/5) & \sin(\pi/5 & \sin(\pi/5) \\
-\sin(\pi/5)& \sin(\pi/5)& \sin(2\pi/5)& \sin(2\pi/5) \\
\sin(\pi/5)& \sin(\pi/5)& \sin(2\pi/5)& \sin(2\pi/5)
\end{bmatrix}
 \cdot
\nonumber \\
&& \ \ \ \ \ \ \ \ \ \ \ \ \ \ \ \ \ \ \ \ \ \ \ \ \ \ \ \ \ \
\begin{pmatrix}
Z(L(-3/5,0);\tau) \\
Z(L(-3/5,3/4);\tau) \\
Z(L(-3/5,1/5);\tau) \\
Z(L(-3/5,-1/20);\tau)
\end{pmatrix}
\label{eqn:mod3}
\end{eqnarray}
and
\begin{eqnarray}
&&
\begin{pmatrix}
Z(L(-3/5,0);\tau+1) \\
Z(L(-3/5,3/4);\tau+1) \\
Z(L(-3/5,1/5);\tau+1) \\
Z(L(-3/5,-1/20);\tau+1)
\end{pmatrix}
\nonumber \\
&=&
\begin{bmatrix}
e^{2\pi i(1/40)} &  0&0&0 \\
0& e^{2\pi i(31/40)} & 0 & 0 \\
0 & 0 & e^{2\pi i(9/40)} & 0 \\
0& 0&0& e^{2\pi i(-1/40)}
\end{bmatrix}
\cdot
\nonumber \\
&& \ \ \  \ \ \ \ \ \ \ \ \ \ \ \ \  \ \ \ \ \ \ \ \ \ \ \ \ \ \ \ \ \ \ \ \ \ \ \ \ \ \ \ \ \ 
\begin{pmatrix}
Z(L(-3/5,0);\tau) \\
Z(L(-3/5,3/4);\tau) \\
Z(L(-3/5,1/5);\tau) \\
Z(L(-3/5,-1/20);\tau)
\end{pmatrix}
.\label{eqn:mod4}
\end{eqnarray}

By (\ref{eq:zv712}),(\ref{eq:zl712}) and  (\ref{eqn:mod1}) -- (\ref{eqn:mod4}), 
we have
\begin{eqnarray*}
&&
\begin{pmatrix}
Z\left(V_{E_{7+1/2}};-1/\tau\right) \\
Z\left(V_{E_{7+1/2}+\alpha_1};-1/\tau\right)
\end{pmatrix}
\nonumber \\
&=&
\frac{2}{\sqrt{5}}
\begin{bmatrix}
\sin(2\pi/5) & \sin(\pi/5) \\
\sin(\pi/5) & -\sin(2\pi/5)
\end{bmatrix}
\cdot
\begin{pmatrix}
Z\left(V_{E_{7+1/2}};\tau\right) \\
Z\left(V_{E_{7+1/2}+\alpha_1};\tau\right)
\end{pmatrix}
\end{eqnarray*}
and
\begin{eqnarray*}
&&
\begin{pmatrix}
Z\left(V_{E_{7+1/2}};\tau+1\right) \\
Z\left(V_{E_{7+1/2}+\alpha_1}; \tau + 1 \right)
\end{pmatrix}
\nonumber \\
&=&
\begin{bmatrix}
e^{2\pi i(-19/60)} & 0 \\
0& e^{2\pi i(29/60)}
\end{bmatrix}
\cdot
\begin{pmatrix}
Z\left(V_{E_{7+1/2}};\tau\right) \\
Z\left(V_{E_{7+1/2}+\alpha_1};\tau\right)
\end{pmatrix}.
\end{eqnarray*}
Hence, especially, the vector space spanned by the characters of the intermediate vertex subalgebra $V_{E_{7+1/2}}$ is invariant under modular transformations.

Furthermore, the rules coincides with that of modular transformations of a system of the solutions of the modular differential equations (\ref{eq:diff1}) with $\mu=551/900$  \cite{MMS2}:
\[
f_1(\tau)=\left( \frac{1}{16}\lambda(1-\lambda)\right)^{(1-x)/6} F\left(\frac{1}{2}-\frac{1}{6}x,\frac{1}{2}-\frac{1}{2}x,1-\frac{1}{3}x,\lambda\right)
\]
and
\[
f_2(\tau)=N\cdot \left( \frac{1}{16}\lambda(1-\lambda)\right)^{(1+x)/6} F\left(\frac{1}{2}+\frac{1}{6}x,\frac{1}{2}+\frac{1}{2}x,1+\frac{1}{3}x,\lambda\right),
\]
where $N$ is a normalization constant, 
$
x=\sqrt{1+36\mu}=1+\frac{1}{2}c
$
and
$
\lambda=\vartheta^4_2(\tau)/\vartheta^4_3(\tau).
$
Here, $\vartheta_i$ $(i=1,\ldots,4)$ are the Jacobi theta functions and $F$ are the hypergeometric functions.
If we multiply the solutions [resp., the characters] by $\eta(\tau)^{38/5}$, we obtain holomorphic modular forms
\[
g_1=\eta(\tau)^{38/5}f_1(\tau) \ \ \  \mbox{and} \ \ \ g_2=\eta(\tau)^{38/5}f_2(\tau)
\]
\[
 [\mbox{resp.,} \ \ \psi_1=\eta(\tau)^{38/5}Z\left(V_{E_{7+1/2}};\tau\right) \ \ \ \mbox{and} \ \ \ \psi_2=\eta(\tau)^{38/5}Z\left(V_{E_{7+1/2}};\tau\right) ]
\]
of weight $19/5$  on $\Gamma(5)$ (with a multiplier system that is $19$ times of that of $\phi_1=\eta(\tau)^{2/5}Z\left(V_{A_{1/2}}\right)$ and $\phi_2=\eta(\tau)^{2/5}Z\left(V_{A_{1/2}+\omega}\right))$.
It is known that the ring of holomorphic modular forms of weight $\frac{1}{5}\mathbb{Z}$ on $\Gamma(5)$ with this multiplier system is the polynomial ring
$\mathbb{C}[\phi_1,\phi_2]$ \cite{BKMS,I,Kan}.
Therefore, $g_1,g_2, \psi_1$ and $\psi_2$ are homogeneous polynomials in $\phi_1$ and $\phi_2$ of degree $19$,
hence $f_1,f_2,Z\left(V_{E_{7+1/2}}\right)$ and $Z\left(V_{E_{7+1/2}+\alpha_1}\right)$ are homogeneous polynomials in $p_1=Z\left(V_{A_{1/2}}\right)$ and $p_2=Z\left(V_{A_{1/2}+\omega}\right)$ of degree $19$.

Then explicitly computing the fist few terms of the $q$-expansions of the homogeneous polynomials $p_1^ip_2^{19-i}$ ($i=0,\ldots,19$), the solutions, and the characters, we obtain
$
f_1(\tau)=Z\left(V_{E_{7+1/2}}\right)$ and $f_2(\tau)=Z\left(V_{E_{7+1/2}+\alpha_1}\right).
$

We show $f_1(\tau)=Z\left(V_{E_{7+1/2}}\right)$.
We have
\[
f_1(\tau)=q^{-19/60}(1+190q+2831q^2+22306q^3+O(q^4))
\]
and
\[
Z\left(V_{E_{7+1/2}};\tau\right)=q^{-19/60}(1+190q+2831q^2+22306q^3+O(q^4)).
\]
Since both the functions $f_1$ and $Z(V_{E_{7+1/2}})$ belong to the ring $q^{-19/60}\mathbb{C}[[q]]$, the functions $f_1$ and $Z(V_{E_{7+1/2}})$ must be linear combinations of the polynomials
$
p_1^{19}$, $p_1^{14}p_2^5$, $p_1^{9}p_2^{10}$ and $ p_1^4p_2^{15}.
$
The first four terms of these polynomials are as follows:
\begin{eqnarray*}
p_1^{19} &=& q^{-19/60}(1+19q+190q^2+1330q^3+O(q^4)), \\
p_1^{14}p_2^5 &=&q^{-19/60}(0+q+14q^2+110q^3+O(q^4)), \\
p_1^9p_2^{10} &=& q^{-19/60}(0+0+q^2+9q^3+O(q^4))
\end{eqnarray*}
and
\[
p_1^4p_2^{15}=q^{-19/60}(0+0+0+q^3+O(q^4)).
\]
Therefore we obtain
\[
f_1(\tau)=p_1^{19}+171p_1^{14}p_2^5+247p_1^9p_2^{10}-57p_1^4p_2^{15}
\]
and
\begin{equation} \label{eq:kanv712}
Z\left(V_{E_{7+1/2}};\tau\right)=p_1^{19}+171p_1^{14}p_2^5+247p_1^9p_2^{10}-57p_1^4p_2^{15}.
\end{equation}
Hence we obtain
$
f_1(\tau)=Z\left(V_{E_{7+1/2}};\tau\right).
$
Similarly, we obtain
\begin{equation} \label{eq:kanl712}
f_2(\tau)=57p_1^{15}p_2^4+247p_1^{10}p_2^9-171p_1^5p_2^{14}+p_2^{19}=Z\left(V_{E_{7+1/2}+\alpha_1}\right).
\end{equation}

Thus we have the assertion that the characters of the intermediate vertex subalgebra $V_{E_{7+1/2}}$ form a basis of the space of the solutions of the modular differential equation (\ref{eq:diff712}).

(Note that the polynomials appearing in (\ref{eq:kanv712}) and (\ref{eq:kanl712}) were studied in \cite{Kan} as a system of the solutions of the differential equation \cite{Kan, Kan2,Kan3}
\begin{eqnarray*}
f''(\tau)-\frac{k+1}{6}E_2(\tau)f'(\tau) + \frac{k(k+1)}{12}E'_2(\tau)f(\tau)=0
\end{eqnarray*}
with $k=19/5$.
The above differential equation is equivalent to the modular differential equation (\ref{eq:diff1}) with $\mu=k(k+2)/36$.
More precisely, $f(\tau)$ satisfies (\ref{eq:diff1}) with $\mu=k(k+2)/36$ if and only if $\eta(\tau)^{2k}f(\tau)$ satisfies the above equation.
Hence we can deduce from (\ref{eq:zv712}) and (\ref{eq:zl712}) the new descriptions of the system of the solutions of the above differential equation with $k=19/5$ using \\
 1) the characters of the intermediate vertex subalgebra $V_{E_{7+1/2}}$ and module or \\
2) the characters of the lattice VOA associated with the $E_7$ root lattice and Virasoro minimal model at $c=-3/5$.)

\begin{rem}
\begin{em}
The lowest weight subspace of the module $V_{E_{7+1/2}+\alpha_1}$ is $57$-dimensional and contains the $56$-dimensional irreducible module of $E_{7}\subset E_{7+1/2}$.
The module $V_{E_{7+1/2}+\alpha_1}$ is expected to give modules of $E_{7+1/2}$ not having been studied.
\end{em}
\end{rem}

\section{Proof of THEOREM \ref{sec:str}}

In this section, we prove the main theorem.
Let us take over the setting and notations in \S 2.

\begin{lem}\label{sec:lem1}
Let $\lambda$ be an element of $L^\circ$. Let $\tau_1,\ldots,\tau_l$ be elements of $L$. Let $i_1,\ldots,i_l$ be integers.
Then, 
\begin{description}
\item[(i)] there exists a unique $g\in \Sy{\ha}$ such that
\begin{equation*}
(e^{\tau_1})_{(i_1+\langle \tau_1,\lambda+\tau_l+\cdots +\tau_2 \rangle )}\ldots(e^{\tau_l})_{(i_l+\langle \tau_l,\lambda \rangle)} \mbox{\boldmath $1$} =g\otimes e^{\tau_1+\cdots +\tau_l}. \label{eqn:lem21}
\end{equation*}
Furthermore,
\item[(ii)] there exists a non-zero constant $r\in \mathbb{C}$ such that
\begin{equation}
(e^{\tau_1})_{i_1}\ldots(e^{\tau_l})_{i_l} e^\lambda =r \cdot g\otimes e^{\lambda+\tau_1+\cdots +\tau_l}. \label{eqn:lem2}
\end{equation}
\end{description}
\end{lem}

\begin{proof}
The assertions follow from the definition of vertex operators for the lattice VOA. Note that if one puts
\[
r=\epsilon(\tau_l,\lambda)\cdot \epsilon(\tau_{l-1},\lambda+\tau_l)\cdots \epsilon(\tau_1,\lambda+\tau_l+\cdots+\tau_2) \in \mathbb{C},
\]
$r$ is non-zero, and (\ref{eqn:lem2}) holds.
\end{proof}

\begin{proof}[Proof of LEMMA \ref{sec:lemmp}]
If $\lambda\in (L(R))^\circ$, the assertion was proved in \cite{MP} (Corollary 4.8.)
The rest of the assertions follow from Lemma \ref{sec:lem1} and that of the case $\lambda=0\in (L(R))^\circ$.
\end{proof}

\begin{lem}\label{sec:cor1prop1}
Let $R'$ and $R''$ be disjoint subsets of $B$. 
Let $\delta'$ be an element of 
 $L_+(R')$
 and $\delta''$ an element of  $L_+(R'')$.
Then
\begin{equation}
(W(R'\sqcup R'';\lambda))^{\delta'+\delta''} \subset \Sy{\ha(R')}\cdot (W(R'';\delta'+\lambda))^{\delta''}. \label{eqn:corprop1eq1}
\end{equation}
\end{lem}

\begin{proof}
Denote the RHS of (\ref{eqn:corprop1eq1}) by X.
We show the case
$R'=\{\beta_1,\ldots,\beta_{n-1}\}$ and $R''=\{\beta_n\}$.
Since $\delta'\in L_+(R')$, $\delta'$ has the form $k_1\beta_1+\cdots+k_{n-1}\beta_{n-1}$ with $k_1,\ldots,k_{n-1}\geq 0$.
Since $\delta''\in L_+(R'')$, $\delta''$ has the form $k_n\beta_n$ with $k_n \geq 0$.
Take $v\in {\cal B}(R'\sqcup R'';\lambda;k_1,\ldots,k_n)$.
By Lemma \ref{sec:lemmp}, it suffices to show $v\in X$.

The element $v$ has the form
$
\varepsilon^{\beta_n}_{\mu_n} \ldots  \varepsilon^{\beta_1}_{\mu_1} e^\lambda
$
with $\mu_i \in M_i(R'\sqcup R'';\lambda;k_1,\ldots,k_n)$ for $1\leq i \leq n$,
 and
$
 \varepsilon^{\beta_{n-1}}_{\mu_{n-1}} \ldots \varepsilon^{\beta_1}_{\mu_1}. e^\lambda
$
has the form
$
g \cdot e^{k_1\beta_1+\cdots+k_{n-1}\beta_{n-1}+\lambda}
$
with  $g\in \Sy{\ha(R')}$.
Then we have
$
v=\varepsilon^{\beta_n}_{\mu_n}. g\cdot e^{\delta'+\lambda}.
$
By the commutation relation
\begin{equation}
[h_k, (e^\alpha)_l]=\langle h,\alpha \rangle (e^\alpha)_{k+l}, \ \ \ \alpha \in L , h\in \mathfrak{h}, k,l \in \mathbb{Z}, \label{eqn:opee}
\end{equation}
of vertex operators of lattice VOAs, it follows that $\varepsilon^{\beta_n}_{\mu_n}. g$ is the sum of operators of the form
$
 h\cdot (e^{\beta_n})_{q_{k_n}} \ldots (e^{\beta_n})_{q_1}
$
with $h\in \Sy{\ha(R')}$ and $q_1,\ldots,q_{k_n} \in \mathbb{Z}$.
Since $\beta_n \in R''$, we have
$
(e^{\beta_n})_{q_{k_n}} \ldots (e^{\beta_n})_{q_1}.e^{\delta'+\lambda} \in (W(R'';\delta'+\lambda))^{k_n\beta_n}.
$
Therefore, $v\in X$.
Thus we have proved the lemma.
\end{proof}

Let $\lambda$ be an element of $L^\circ$.
Let $\gamma$ be an element of $L_+(R)$
and $\delta$ an element of $L(S)$.

For $\mu \in L(R,S)$,
consider the set 
\begin{eqnarray*}
T(\mu;\lambda) &=& \{ f\cdot (e^{\tau_l})_{i_l}\ldots(e^{\tau_1})_{i_1}.e^{\tau_0+\lambda} |
f\in \Sy{\ha(S)}, l\geq 0, i_1,\ldots,i_l\in \mathbb{Z}, \\
&& \ \ \ \ \ \tau_0\in L(S), \tau_1,\ldots,\tau_l \in R\sqcup S \sqcup (-S) \ \ \mbox{with} \\
&& \ \ \ \ \ \tau_0+\tau_1+\cdots +\tau_l=\mu
\}.
\end{eqnarray*}
Here, $- S=\{-\sigma \in L | \sigma \in S\}$.

\begin{lem}\label{sec:lemmonomial}
The set $T(\delta+\gamma;\lambda)$  
spans the vector space $W(R,S;\lambda)^{\delta+\gamma}$.
\end{lem}

\begin{proof}
For $\tau\in L(S)$,
we have
$
e^{\tau+\lambda}= \varepsilon (\tau,\lambda)(e^\tau)_{-1-\langle \tau,\lambda \rangle}e^\lambda.
$
Then it follows that each element of $T(\delta+\gamma;\lambda)$ belongs to $W(R,S;\lambda)^{\delta+\gamma}$.

Consider the subset
\begin{eqnarray*}
&& \{ f\cdot (e^{\tau_l})_{i_l}\ldots(e^{\tau_1})_{i_1}.e^{\lambda} |
f\in \Sy{\ha(S)}, l\geq 0, i_1,\ldots,i_l\in \mathbb{Z}, \\
&& \ \ \ \ \ \tau_1,\ldots,\tau_l \in R\sqcup S \sqcup (-S) \ \ \mbox{with} \ \ 
\tau_1+\cdots +\tau_l=\delta+\gamma
\}
\end{eqnarray*}
of $T(\delta+\gamma;\lambda)$.
By (\ref{eqn:opee}),
 it follows that the subset spans $W(R,S;\lambda)^{\delta+\gamma}$.
Hence the set $T(\delta+\gamma;\lambda)$ spans $W(R,S;\lambda)^{\delta+\gamma}$.
\end{proof}

\begin{prop}\label{sec:thm1}
\begin{equation*}
W(R,S;\lambda)^{\delta+\gamma}=\Sy{\ha(S)} \cdot W(R;\delta+\lambda)^{\gamma}. \label{eqn:thm1}
\end{equation*}
\end{prop}

\begin{proof}
We show the case $\lambda=0$.
That is, we show
$
W(R,S)^{\delta+\gamma}=\Sy{\ha(S)}\cdot W(R;\delta)^{\gamma}.
$
Denote the LHS by $X$ and the RHS by $Y$.
Note that both $X$ and $Y$ are free $\Sy{\ha(S)}$-modules.
We have $X \supset Y$,
since $W(R;\delta)\subset W(R,S)$.

Let us show $X \subset Y$.
Take $v\in T(\delta+\gamma;0)$.
By Lemma \ref{sec:lemmonomial}, it suffices to show
$
v\in Y.
$
 The element $v$ has the form 
$
f \cdot (e^{\tau_l})_{i_l}\ldots(e^{\tau_1})_{i_1}. e^{\tau_0}
$
with $f\in \Sy{\ha(S)}$, $l\geq 0$, $i_1,\ldots,i_l\in \mathbb{Z}$, $\tau_0 \in L(S)$ and $\tau_1,\ldots,\tau_l \in R\sqcup  S \sqcup (-S)$ with $\tau_0+\tau_1+\cdots +\tau_l=\delta+\gamma$.

If $\tau_1,\ldots,\tau_l\in R$, then $\tau_0=\delta$ and $\tau_1+\cdots+\tau_l=\gamma$, since $R\cap (S\sqcup (-S))=\emptyset$. 
Then it follows that $v\in Y$.

So assume that $j \in \{1,\ldots,l\}$ is the minimum number satisfying $\tau_j\not \in R$.
Then $\tau_j$ has the form
$
\tau_j=t\cdot \sigma_k
$
with $t\in \{\pm 1\}$ and $k\in \{1,\ldots,s\}$.
Put $R'=t\cdot S$ and $R''=R$.
Here, $t\cdot S=\{t\cdot \sigma_1,\ldots,t\cdot \sigma_s\}$.
Then $R'$ and $R''$ are disjoint.
Consider the vector
$
v'= (e^{\tau_j})_{i_j}  \cdots (e^{\tau_1})_{i_1} e^{\tau_0}.
$
(Note that $v=f\cdot (e^{\tau_l})_{i_l}\ldots(e^{\tau_{j+1}})_{i_{j+1}}.v'$.)
Then $v'\in (W(R'\sqcup R'';\tau_0))^{\tau_1+\cdots+\tau_j}$, since $\tau_j \in R'$ and $\tau_1,\ldots,\tau_{j-1} \in R''$.
Moreover, $\tau_j\in L_+(R')$, and $\tau_1+\cdots+\tau_{j-1}\in L_+(R'')$.
From Lemma \ref{sec:cor1prop1} and the relation $\Sy{\ha(t\cdot S)}=\Sy{\ha(S)}$, it follows that
$v' \in \Sy{\ha(S)}\cdot (W(R;\tau_0+\tau_j))^{\tau_1+\cdots+\tau_{j-1}}.
$
Therefore $v'$ is the sum of elements of the form
\[
g\cdot (e^{\tau'_{j-1}})_{i'_{j-1}}\ldots(e^{\tau'_1})_{i'_1}.e^{\tau_0+\tau_j}
\]
with $g\in\Sy{\ha(S)}$, $i'_1,\ldots,i'_{j-1}\in \mathbb{Z}$ and $\tau'_1,\ldots,\tau'_{j-1}\in R$ 
with $\tau'_1+\cdots+\tau'_{j-1}=\tau_1+\cdots+\tau_{j-1}$.
Therefore, to show $v\in Y$, it suffices to show that the vector
$
f \cdot (e^{\tau_l})_{i_l}\ldots(e^{\tau_{j+1}})_{i_{j+1}}.g\cdot (e^{\tau'_{j-1}})_{i'_{j-1}}\ldots(e^{\tau'_1})_{i'_1}.e^{\tau_0+\tau_j}
$
belongs to $Y$.

By (\ref{eqn:opee}), the operator $(e^{\tau_l})_{i_l}\ldots(e^{\tau_{j+1}})_{i_{j+1}}.g$ is the sum of operators of the form
$
h\cdot (e^{\tau_l})_{p_l}\cdots (e^{\tau_{j+1}})_{p_{j+1}}
$
with $h\in \Sy{\ha(S)}$ and $p_{j+1},\ldots,p_l \in \mathbb{Z}$.
Therefore, to show $v\in Y$, it suffices to show that the vector
\[
w=f\cdot h\cdot (e^{\tau_l})_{p_l}\ldots(e^{\tau_{j+1}})_{p_{j+1}}.(e^{\tau'_{j-1}})_{i'_{j-1}}\ldots(e^{\tau'_1})_{i'_1}.e^{\tau_0+\tau_j}
\]
belongs to $Y$.

Since $\tau_j \in S\sqcup (-S)$, the vector $w$ belongs to $T(\delta+\gamma;0)$ as the vector $v$.
Since $\tau_j \in S\sqcup (-S)$ and $\tau'_1,\ldots,\tau'_{j-1} \in R$, the number of elements of the set 
\[
\{\tau| \tau\not \in R, \tau\in\{\tau'_1,\ldots,\tau'_{j-1},\tau_{j+1},\ldots,\tau_l\} \}
\]
is fewer than the number of elements of the set
$
\{\tau| \tau\not \in R, \tau \in \{\tau_1,\ldots,\tau_l\}\}.
$
Thus, repeating this procedure, we have $v\in Y$, and the proof is complete.
\end{proof}

\begin{proof}[Proof of THEOREM \ref{sec:str}]
The assertion follows from Proposition \ref{sec:thm1} and the fact that the set of the charges of $W(R,S;\lambda)$ agrees with $L(R,S)$.
\end{proof}

\subsection*{Acknowledgments}
The author wishes to express his thanks to his advisor, Professor A.\ Matsuo for helpful advice and kind encouragement.
He also wishes to express his thanks to H.\ Yamauchi, H.\ Shimakura and M.\ Okumura for helpful conversations.

Note added: After finishing this work, we learned of a recent related work by Kaneko, Nagatomo and Sakai \cite{Kan4},
where the Kaneko-Zagier equations and the Mathur-Mukhi-Sen classification are studied in detail.


\begin{thebibliography}{100}

\bibitem{A}
Andrews, G.\ E.: The theory of partitions. Encyclopedia of Mathematics and Its Applications, Vol. 2. Addison-Wesley Publishing Co., Reading, Mass.-London-Amsterdam (1976)

\bibitem{A2}
Andrews, G.\ E.: $q$-Series: Their development and application in analysis, number theory, combinatorics, physics, and computer algebra,
  Regional Conf.\ Ser.\ in Math., no.\ {\bf 66} (1986)

\bibitem{AKS}
Ardonne, E, Kedem, R., Stone, M.: Fermionic characters of arbitrary highest-weight integrable $sl_{r+1}$-modules, Comm.\  Math.\ Phys.\ {\bf 264} 427-464 (2006)

\bibitem{BKMS}

Bannai, E., Koike, M., Munemasa, A., Sekiguti, J.: Klein's icosahedral equation and modular forms, preprint (1999)

\bibitem{Bax}
Baxter, R.\ J.: Exactly Solved Models in Statistical Mechanics, Academic Press, London and New York, (1982)

\bibitem{B}
Borcherds, R.\ E.: Vertex algebras, Kac-Moody algebras, and the Monster.
Proc.\ Nat.\ Acad.\ Sci.\ U.S.A.\ {\bf 83}, no. 19, 3068-3071 (1986)

\bibitem{CalLM3}
Calinescu, C., Lepowsky, J., Milas, A.:
Vertex-algebraic structure of the principal subspaces of level one modules for the untwisted affine Lie algebras of type A, D, E.
J.\ Algebra {\bf 323}, no.\ 1, 167-192 (2010)

\bibitem{CdM}
Cohen, A.\ M., de Man, R.: Computational evidence for Deligne's conjecture regarding exceptional Lie groups, C.\ R.\ Acad.\  Sci.\ Paris S\'{e}r.\ I Math.\ {\bf 322} 427-432 (1996)

\bibitem{CLM2}
Capparelli, S., Lepowsky, J., Milas A.: The Rogers-Selberg recursions, the Gordon-Andrews identities and intertwining operators, The Ramanujan Journal
{\bf 12}  379-397 (2006)

\bibitem{CoLM}
Cook, W.\ J., Li, H., Misra, K.\ C.: A recurrence relation for characters of highest weight integrable modules for affine Lie algebras, Comm.\ in Contemporary Math.\ 
{\bf 9},  no.\ 2, 121-133 (2007)

\bibitem{D}
Deligne, P.: La s\'{e}rie exceptionalle de groupes de Lie, C.\ R.\ Acad.\ Sci.\ Paris S\'{e}r.\ I Math.\ {\bf 322} 321-326 (1996)

\bibitem{FFJMM}
Feigin, B., Feigin, E., Jimbo, M., Miwa, T., Mukhin, E.: Principal $\widehat{\mathfrak{sl}_3}$ subspaces and quantum Toda Hamiltonian, arXiv:0707.1635. (2007)

\bibitem{FS}
Feigin, B.\ L., Stoyanovsky, A.\ V.: Quasi-particles models for the representations of Lie algebras and geometry of flag manifold.
arXiv:hep-th/9308079. (1993)

\bibitem{G}
Georgiev, G.\ N.: Combinatorial constructions of modules for infinite-dimensional Lie algebras, I. Principal subspace, J.\  Pure Appl.\ Algebra {\bf 112},  247-286 (1996)

\bibitem{GZ1}
Gel'fand, I.\ M., Zelevinsky, A.\ V.: Models of representations of classical groups and their hidden symmetries, (Russian) Funkt.\ Anal.\ i Prilozhen.\ 18 (3) 14-31 (1984)

\bibitem{I}
Ibukiyama, T.: Modular forms of rational weights and modular varieties, Abh.\ Math.\ Sem.\ Univ.\ Hamburg, {\bf 70} 315-339 (2000)

\bibitem{K2}
Kac, V.\ G.: Infinite dimensional Lie algebras. Third Edition. Cambridge University Press, Cambridge, UK (1990)

\bibitem{Kan}
Kaneko, M.: On modular forms of weight $(6n+1)/5$ satisfying a certain differential equation, Number Theory: Tradition and Modernization 97-102 (2006)

\bibitem{Kan2}
Kaneko, M., Koike, M.: On modular forms arising from a differential equations of hypergeometric type. Ramanujan J.\ {\bf 7}  145-164 (2003)

\bibitem{Kan4}
Kaneko, M., Nagatomo, K., Sakai, Y.: Modular forms and second order ordinary differential equations: applications to vertex operator algebras. Lett.\ Math.\ Phys.\ {\bf 103} 439-453 (2013)

\bibitem{Kan3}
Kaneko, M., Zagier, D.: Supersingular $j$-invariants, hypergeometric series, and Atkin's orthogonal polynomials. AMS/IP Stud.\ Adv.\ Math.\ {\bf 7}  97-126 (1998)

\bibitem{KKMM}
Kedem, R., Klassen, T.~R., McCoy, B.~M.,  Melzer, E.: Fermionic sum representations for conformal field theory characters,
Phys.\ Lett.\ B vol.\ {\bf 307} issue 1-2, 68-76 (1993)

\bibitem{LaM2}
Landsberg, J.\ M., Manivel, L.: Triality, exceptional Lie algebras, and Deligne dimension formulas, Adv.\ Math.\ {\bf 171} 59-85 (2002) 

\bibitem{LaM1}
Landsberg, J.\ M., Manivel, L.: The sextonions and $E_{7\frac{1}{2}}$, Adv.\ Math.\ {\bf 201} (1)  143-179 (2006)

\bibitem{Mat}
Matsuo, A.: Norton's trace formulae for the Griess algebra of a vertex operator algebra with larger symmetry, Commun.\ Math.\  Phys.\ {\bf 224} 565-591 (2001)

\bibitem{M}
Milas, A.: Ramanujan's ``Lost Notebook" and the Virasoro algebra, Comm.\ Math.\ Phys.\ vol.\ {\bf 251}, no.\ 3, 567-588 (2004)

\bibitem{MP}
Milas, A., Penn, M.: Lattice vertex algebras and combinatorial bases: general case and ${\cal W}$-algebras. 
New York J.\ Math.\ {\bf 18} 621-650 (2012)

\bibitem{MMS}
Mathur, S., Mukhi, S., Sen, A.:
On the classification of rational conformal field theories,
Phys.\ lett.\ B, Vol.\ {\bf 213}, Issue.\ 3, 303-308 (1988)

\bibitem{MMS2}
Mathur, S., Mukhi, S., Sen, A.:
Reconstruction of conformal field theories from modular geometry on the torus,
Nucl.\ Phys.\ B, Vol.\ {\bf 318} 483-540 (1989)

\bibitem{P1}
Primc, M.: Vertex operator construction of standard modules for $A_n^{(1)}$, Pacific J.\ Math.\ {\bf 162} 143-187 (1994)

\bibitem{S1}
Shtepin, V.\ V.: Intermediate Lie algebras and their finite-dimensional representations.
Russian Acad.\ Sci.\ Izv.\ Math.\ Vol.\ {\bf 43}, No.3 559-579 (1994)

\bibitem{SF}
Stoyanovskii, A.\ V., Feigin, B.\ L., Functional models for representations of current algebras and semi-infinite Schubert cells.
 Funct.\ Anal.\ Appl.\ {\bf 28}, no.\ 1, 55-72 (1994)

\bibitem{T2}
Tuite, M.\ P.: Exceptional vertex operator algebras and the Virasoro algebra, Contemp.\ Math.\ {\bf 497} 213-225 (2009)

\bibitem{W}
Wakimoto, M.: Lectures on infinite dimensional Lie algebras. World Scientific Publishing, Singapore (2001)

\end{thebibliography}
\end{document}